\documentclass[a4paper,11pt,leqno]{amsart}
\usepackage{amsmath,amsthm,amsfonts,amssymb}

\newcommand{\bD}{\mathbb{D}}
\newcommand{\bC}{\mathbb{C}} 
\newcommand{\diam}{\operatorname{diam}}
\newcommand{\bR}{\mathbb{R}}
\newcommand{\rd}{\mathrm{d}}
\newcommand{\bN}{\mathbb{N}}
\newcommand{\bP}{\mathbb{P}}
\newcommand{\bB}{\mathbb{B}}

\theoremstyle{theorem}
\newtheorem{theorem}{Theorem}[section]
\newtheorem{mainth}{Theorem}

\newtheorem{corollary}[theorem]{Corollary}

\theoremstyle{definition}
\newtheorem{definition}[theorem]{Definition}

\newtheorem*{acknowledgement}{Acknowledgement}
     
\theoremstyle{remark}
\newtheorem{remark}[theorem]{Remark} 
\newtheorem{notation}[theorem]{Notation} 
\newtheorem{fact}[theorem]{Fact} 

\newtheorem{example}[theorem]{Example}

\numberwithin{equation}{section}

\usepackage{color}

\begin{document} 

\title[Isolated essential singularity of a holomorphic curve]{
A Lehto--Virtanen-type theorem and a rescaling principle for 
an isolated essential singularity of a holomorphic curve
in a complex space}

\date{\today}

\author{Y\^usuke Okuyama}
\address{Division of Mathematics, Kyoto Institute of Technology, Sakyo-ku, Kyoto 606-8585, Japan}
\email{okuyama@kit.ac.jp}
\thanks{Partially supported by JSPS Grant-in-Aid for Young Scientists (B), 24740087.}

\dedicatory{To the memory of Professor Shoshichi Kobayashi}



\begin{abstract}
 We establish a Lehto--Virtanen-type theorem and a rescaling principle
 for an isolated essential singularity of a holomorphic curve in a complex space,
 which are useful for establishing a big Picard-type theorem 
 and a big Brody-type one for holomorphic curves.
\end{abstract}

\maketitle

\section{Introduction}\label{sec:intro}

Let $V$ be a complex space.
For a holomorphic mapping $f:\bD\setminus\{0\}\to V$,
we say that $f$ has an {\itshape isolated essential singularity} at the origin
if $f$ does not extend to a holomorphic mapping from $\bD$ to a complex space. 
Our aim is to contribute to the study of the Kobayashi hyperbolicity and
the Brody one of a complex space by
establishing a Lehto--Virtanen-type theorem and a rescaling principle
for an isolated essential singularity
of a holomorphic curve in a complex space.

\begin{notation}
 Set $\bD(a,r):=\{z\in\bC:|z-a|<r\}$ for every $a\in\bC$ and every $r>0$,
 and set $\bD(r):=\bD(0,r)$ for every $r>0$. Then $\bD(1)=\bD$.
 For every metric $\delta$ on a complex space $V$, 
 set 
 $\diam_{\delta}(S):=\sup\{\delta(a,a'):a,a'\in S\}$
 for a non-empty subset $S$ in $V$. 
 Finally, for a complex space $V$, let $d_V$ be the Kobayashi pseudometric on $V$.
\end{notation}

For the foundation of hyperbolic complex spaces,
see the books \cite{Kobayashi98,Kobayashi05}.

\subsection{A Lehto--Virtanen-type theorem and a big Picard-type theorem}
The following is a generalization of Lehto--Virtanen \cite{LV57Fennicae}.

\begin{mainth}[a Lehto--Virtanen-type theorem]\label{th:HeinonenRossi}
Let $V$ be a complex space equipped with a metric $\delta$
inducing the topology of $V$,
and $f:\bD\setminus\{0\}\to V$ be a holomorphic mapping
having an isolated essential singularity at the origin.
If $\bigcap_{r>0}\overline{f(\bD(r)\setminus\{0\})}\neq\emptyset$,
then there exists
a sequence $(z_n)$ in $\bD\setminus\{0\}$ tending to
$0$ as $n\to\infty$ such that $\lim_{n\to\infty}f(z_n)$ exists
in $V$ and that $\liminf_{n\to 0}\diam_{\delta}(f(\partial\bD(|z_n|)))>0$.
\end{mainth}

\begin{remark}
 There always exists a metric $\delta$ on $V$ inducing the topology of $V$.
 When $V$ is Kobayashi hyperbolic, we can set $\delta=d_V$ 
 in Theorem \ref{th:HeinonenRossi}.
\end{remark}

In this article,
we do {\itshape not} require the relative compactness of $Y$ in $Z$
in the following definition.
  
\begin{definition}[a hyperbolically imbedded complex subspace]
 A complex subspace $Y$ in a complex space $Z$ is called a 
 hyperbolically imbedded complex subspace in $Z$ if  
 for every distinct points $p,q\in\overline{Y}$, there exist
 open neighborhoods $U_p,U_q$ of $p,q$ in $Z$, respectively, such that
 $d_Y(U_p\cap Y,U_q\cap Y)>0$.
\end{definition}

The following, which will be established as an application of
Theorem \ref{th:HeinonenRossi},
slightly generalizes Kobayashi {\cite[Theorem VII.3.6]{Kobayashi05}}.
 
\begin{mainth}[a big Picard-type theorem]\label{th:PicardKobayashi}
 Let $Y$ be a hyperbolically imbedded complex subspace in a complex space $Z$,
and $f:\bD\setminus\{0\}\to Y$ be a holomorphic mapping. 
 If $\bigcap_{r>0}\overline{f(\bD(r)\setminus\{0\})}\neq\emptyset$
 $($as a subset in $Z)$,
then $f$ extends to a holomorphic mapping from $\bD$ to $Z$.
\end{mainth}

A Kobayashi hyperbolic complex space is a hyperbolically imbedded
complex subspace in itself.
Hence Theorem \ref{th:PicardKobayashi} is a generalization
of Kwack \cite{Kwack69}, which can be restated as follows.

\begin{theorem}[{Kwack \cite{Kwack69}}]\label{th:Picard}
 If there is a holomorphic mapping from $\bD\setminus\{0\}$
 to a complex space $V$ having an isolated
 essential singularity at the origin and satisfying
 $\bigcap_{r>0}\overline{f(\bD(r)\setminus\{0\})}\neq\emptyset$,
 then $V$ is not Kobayashi hyperbolic. 
\end{theorem}

\begin{remark}
 The corresponding little Picard-type theorem, which does not require
 the assumption $\bigcap_{r>0}\overline{f(\bD(r)\setminus\{0\})}\neq\emptyset$,
 follows from the {\itshape definition} of
 the Kobayashi hyperbolicity: see Fact \ref{th:little} below.
\end{remark}

\subsection{A rescaling principle and a big Brody-type theorem}

When a complex space is compact, it admits a nice metric
(see Theorem \ref{th:Yamanoi}).

Theorem \ref{th:HeinonenRossi} together with
a Zalcman-type theorem (see Theorem \ref{th:ZB} or \ref{th:Zalcman}) also 
establishes the following rescaling principle for an isolated essential singularity
of a holomorphic curve.

\begin{mainth}[a rescaling principle]\label{th:rescaling}
Let $V$ be a compact complex space
equipped with a metric satisfying the conditions in Theorem $\ref{th:Yamanoi}$, and
$f:\bD\setminus\{0\}\to V$ be a holomorphic mapping. Then, 
$f$ has an isolated essential singularity at the origin if and only if
there are 
sequences $(z_k)$ and $(\rho_k)$ 
in $\bC$ and $(0,\infty)$, respectively, 
and a non-constant holomorphic mapping $g:X\to V$, 
where $X$ is either $\bC$ or $\bC\setminus \{0\}$, 
such that 
$\lim_{k\to\infty}z_k=0$ on $\bC$, that $\lim_{k\to\infty}\rho_k=0$ in $\bR$, and
that $\lim_{k\to\infty}f(z_k+\rho_k v)= g(v)$
locally uniformly on $X$. 
\end{mainth}

\begin{remark}
Originally, similar results to Theorems \ref{th:HeinonenRossi} 
and \ref{th:rescaling} have been established in 
\cite[Lemma 3.1,\ Theorem 1]{OPrescaling} 
for quasiregular mappings from a punctured ball to a (compact) Riemannian manifold 
having an isolated essential singularity at the puncture, 
and are applied not only to deduce ``big'' results from 
their corresponding ``little'' ones (e.g., Holopainen--Pankka's big Picard-type
theorem \cite{HP05} from Holopainen--Rickman's little one
\cite{HR92} for quasiregular mappings when the target is compact) but also to
establish the density of repelling periodic points in the Julia set
for {\itshape local uniformly quasiregular} dynamics 
including complex dynamics \cite{OPrepelling}. 
\end{remark}

\begin{remark}[a holomorphic mapping exceptional in Julia's sense]\label{th:Julia}
 In Theorem \ref{th:rescaling},
 when $V$ is a compact Hermitian complex manifold equipped with an Hermitian metric $\delta_V$,
 the sequence $(z_k)$ can identically equal $0$ (and then $X=\bC\setminus\{0\}$) if
 the mapping $f$ is {\itshape exceptional in Julia's sense} 
 in that
 \begin{gather*}
 \limsup_{z\to 0}|z|f^\#(z)<\infty,
 \end{gather*}
 where we set $f^\#(z):=\lim_{w\to z}\delta_V(f(z),f(w))/|z-w|$ 
 on $\bD\setminus\{0\}$.
 Conversely, if $\limsup_{z\to 0}|z|f^\#(z)=\infty$, then 
 the case $X=\bC$ can occur. 
 Some examples of both non-exceptional $f$ and exceptional $f$
 in Julia's sense have been known and 
studied in the Nevanlinna theory: see Lehto--Virtanen
 \cite{LV57Fennicae}.
\end{remark}

When $V$ is compact, the following improvement of Theorem \ref{th:Picard}
immediately follows from Theorem \ref{th:rescaling} 
(see also Remark \ref{th:Brody}). 

\begin{corollary}[a big Brody-type theorem]\label{th:bigBrody}
If there is a holomorphic mapping from $\bD\setminus\{0\}$
to a compact complex space $V$ having an isolated
essential singularity at the origin,
then $V$ is not Brody hyperbolic. 
\end{corollary}

\begin{remark}
Corollary \ref{th:bigBrody} 
is also a consequence of Theorem \ref{th:Picard}
and the equivalence between the Kobayashi hyperbolicity and the Brody one
for compact complex spaces, the latter of which is known as 
Brody's theorem \cite{Brody78}. See Fact \ref{th:little} and Remark \ref{th:Brodyproof} below.  
\end{remark}

\subsection{Organization of this article}
In Section \ref{sec:background}, we gather some background materials
from the Kobayashi hyperbolic geometry
as well as a proof of a Zalcman-type theorem (Theorem \ref{th:ZB}) using
the original Zalcman's lemma (Theorem \ref{th:Zalcman})
in the setting that $V$ is not necessarily an Hermitian
compact complex manifold,
which we hope might be of independent
interest.
We show Theorems \ref{th:HeinonenRossi}, \ref{th:PicardKobayashi},
and \ref{th:rescaling} in
Sections \ref{sec:LV}, \ref{sec:Picard}, and \ref{sec:rescaling}, respectively.
Section \ref{sec:Julia} is
devoted to some details on Remark \ref{th:Julia}.

\section{Background}\label{sec:background}

Let $V$ be a complex space.
Let us recall the definition of the Kobayashi pseudometric $d_V$ on $V$
(cf.\ \cite[\S VII]{Kobayashi05} or \cite[\S 2]{Yamanoisurvey}).

\begin{definition}[the Kobayashi pseudometric $d_V$ on $V$]\label{th:definitionKobayashi}
 For every $p,q\in V$,
 a chain of holomorphic disks from $p$ to $q$ is 
 a pair of a finite sequence $(f_j)_{j=1}^m$ of holomorphic mappings from $\bD$
 to $V$ and a finite sequence $((a_j,b_j))_{j=1}^m$ in $\bD\times\bD$ such that
 $f_1(a_1)=p$, $f_m(b_m)=q$ and that for every $j\in\{1,2,\ldots,m-1\}$,
 $f_j(b_j)=f_{j+1}(a_{j+1})$ (the $m\in\bN$ depends on each
 chain of holomorphic disks from $p$ to $q$).
 Let $\rho_{\bD}$ be the Poincar\'e (, Bergman or hyperbolic) metric on $\bD$. Then
 $d_V(p,q)$ is defined by the infimum
 of the sum $\sum_{j=1}^m\rho_{\bD}(a_j,b_j)$, where the infimum is taken over
 all chains $((f_j)_{j=1}^m,((a_j,b_j))_{j=1}^m)$
 of holomorphic disks from $p$ to $q$.
\end{definition}

\begin{example}\label{th:disk}
 For every $a\in\bC$ and every $R>0$,
 the Kobayashi pseudometric $d_{\bD(a,R)}$ on $\bD(a,R)$ coincides with the 
 Poincar\'e metric $\rho_{\bD(a,R)}$ on $\bD(a,R)$,
 which is a K\"ahler metric given by $\rd s_{\bD(a,R)}=R|\rd z|/(R^2-|z-a|^2)$
 on $\bD(a,R)$. The following comparison is useful:
 for every $a\in\bC$, every $R>0$, every $r\in(0,R)$,
 and every $z,w\in\overline{\bD(a,r)}$,
\begin{gather}
 \frac{1}{R}|z-w|\le d_{\bD(a,R)}(z,w)
 \le\frac{R}{R^2-r^2}|z-w|.\label{eq:comparable}
\end{gather}
\end{example}

\begin{example}[{cf.\ \cite[Propositions IV.1.1 and VI.2.1]{Kobayashi05}}]\label{th:puncture} 
The Kobayashi pseudometric $d_{\bD\setminus\{0\}}$ on $\bD\setminus\{0\}$
coincides with the hyperbolic metric on $\bD\setminus\{0\}$,
which is a K\"ahler metric 
given by $\rd s_{\bD\setminus\{0\}}=|\rd z|/(-|z|\log|z|)$
on $\bD\setminus\{0\}$.
In particular, the arc length of the circle $\partial\bD(r)$ with respect to
$d_{\bD\setminus\{0\}}$ is $O(1/(-\log r))$ as $r\to 0$.
\end{example}

\begin{fact}\label{th:decreasing}
The Kobayashi pseudometrics on complex spaces 
enjoy the {\itshape non-increasing} property 
under holomorphisms in that for complex spaces $X,Y$,
a holomorphic mapping $f:X\to Y$, and points $x,x'\in X$,
\begin{gather*}
 d_Y(f(x),f(x'))\le d_X(x,x'). 
\end{gather*}
In particular,
the Kobayashi pseudometrics are invariant under biholomorphisms
between complex spaces.
\end{fact}

\begin{definition}[the Kobayashi hyperbolicity and the Brody one]
 A complex space $V$ is said to be Kobayashi hyperbolic (resp.\ Brody hyperbolic)
 if the Kobayashi pseudometric $d_V$ is a metric on $V$
 (resp.\ if there is no non-constant holomorphic mapping from 
 $\bC$ to $V$). 
\end{definition}

\begin{fact}\label{th:topology}
 If a complex space $V$ is Kobayashi hyperbolic, then $d_V$ induces 
the topology of $V$.
\end{fact}

\begin{remark}\label{th:Brody}
 For a non-constant holomorphic mapping $g:\bC\setminus\{0\}\to V$, 
 $g\circ\exp:\bC\to V$ is non-constant and holomorphic.
 Conversely, for a non-constant
 holomorphic mapping $g:\bC\to V$, 
 $g|(\bC\setminus\{0\}):\bC\setminus\{0\}\to V$ is  
 non-constant and holomorphic. 
Hence,
{\itshape  a complex space $V$ is Brody hyperbolic if and only if
 there is no non-constant holomorphic mapping from $X$ to $V$,
 where $X$ is either $\bC$ or $\bC\setminus\{0\}$.}
\end{remark}

\begin{fact}[a little Picard-type theorem and Brody's theorem]\label{th:little}
 {\itshape If a complex space $V$ is Kobayashi hyperbolic, then 
 it is also Brody hyperbolic}: this
 ``little Picard-type theorem''is almost by the definition of the Kobayashi
 hyperbolicity and that of the Brody one.

 Brody's theorem asserts that {\itshape the converse
 is also true if in addition $V$ is compact,} that is,
 {\itshape a Brody hyperbolic compact complex space is Kobayashi hyperbolic}.
\end{fact}

When a complex space is compact, it admits a nice metric.

\begin{theorem}[cf.\ {\cite[Subsection 4.1]{Yamanoisurvey}}]\label{th:Yamanoi}
For every compact complex space $V$,
there is a metric $\delta$ on $V$ satisfying that
\begin{enumerate}
 \item the distance $\delta$ induces the topology of $V$, and that
 \item there is an open covering $\{U_x:x\in V\}$ of $V$ such that
       for every $x\in V$, $U_x$ is a Kobayashi hyperbolic subdomain in $V$
       containing $x$ and satisfies $\delta\le d_{U_x}$ on $U_x(\times U_x)$.
       \label{item:neighborhood}
\end{enumerate}
\end{theorem}

The following {\itshape Lipschitz continuity on disks} of holomorphic curves
from a domain in $\bC$ into compact complex spaces is useful:
the boundedness \eqref{eq:Lipschitz} is a special case of
the ``only if'' part of a Marty-type theorem
(Theorem \ref{th:Marty}) below, and the equality \eqref{eq:invariance}
is by the invariance under biholomorphisms of the Kobayashi (pseudo)metrics
(cf.\ Fact \ref{th:decreasing}).

\begin{theorem}[cf.\ {\cite[Subsection 2.3]{Yamanoisurvey}}]
 Let $V$ be a compact complex space equipped with a metric $\delta$
 satisfying the conditions in Theorem $\ref{th:Yamanoi}$.
Then for every open disk $\bD(a,r)$ and
every holomorphic mapping $f$ from an open neighborhood of
$\overline{\bD(a,r)}$ in $\bC$ to $V$, we have
\begin{gather}
 L_{f,\bD(a,r)}:=\sup_{w,w'\in\bD(a,r),w\neq w'}
\frac{\delta(f(w),f(w'))}{d_{\bD(a,r)}(w,w')}<\infty,\label{eq:Lipschitz}
 \end{gather}
which satisfies the invariance
\begin{gather}
 L_{f\circ\phi,\bD(b,s)}=L_{f,\bD(a,r)}\label{eq:invariance}
\end{gather}
for every biholomorphism $\phi:\bD(b,s)\to\bD(a,r)=\phi(\bD(b,s))$.
\end{theorem}

\begin{definition}
 For complex spaces $X$ and $Y$,
 let $\mathcal{O}(X,Y)$ be the set of all holomorphic mappings
 from $X$ to $Y$.
\end{definition}

We conclude this section with 
a generalization of Marty's theorem \cite[Th\'eor\`eme 5]{Marty31}
and that of Zalcman's lemma \cite{Zalcman75}, both of which
characterize the (non-)normality on $D$ of a family in $\mathcal{O}(D,V)$
for a domain $D$ in $\bC$ and a compact complex space $V$.

\begin{theorem}[a Marty-type theorem]\label{th:Marty}
 Let $D$ be a domain in $\bC$, $V$ a compact complex space
 equipped with a metric $\delta$
 satisfying the conditions in Theorem $\ref{th:Yamanoi}$,
 and $\mathcal{F}$ a family in $\mathcal{O}(D,V)$. Then, 
 $\mathcal{F}$
is normal on $D$ if and only if $\sup_{f\in\mathcal{F}}L_{f,\bD(a,r)}<\infty$
 for every $a\in D$ and every $r>0$ satisfying $\bD(a,r)\Subset D$. 
\end{theorem}

\begin{proof}
 By the Arzel\`a--Ascoli theorem, $\mathcal{F}$ 
 is normal on $D$ if and only if $\mathcal{F}$ is locally equicontinuous on $D$.
 Let us show that $\mathcal{F}$ is locally equicontinuous on $D$
 if and only if $\sup_{f\in\mathcal{F}}L_{f,\bD(a,r)}<\infty$
 for every $a\in D$ and every $r>0$ satisfying $\bD(a,r)\Subset D$.
 The ``if'' part is straightforward by \eqref{eq:comparable}.

 Suppose that $\bD(a,r)\Subset D$ but
 $\sup_{f\in\mathcal{F}}L_{f,\bD(a,r)}=\infty$
 for some $a\in D$ and some $r>0$.
 Then there are a sequence
 $(f_k)$ in $\mathcal{F}$ and sequences $(z_k),(w_k)$ in $\bD(a,r)$
 such that $z_k\neq w_k$ for every $k\in\bN$ and that
 \begin{gather}
  \lim_{k\to\infty}\frac{\delta(f_k(z_k),f_k(w_k))}{d_{\bD(a,r)}(z_k,w_k)}=\infty.\label{eq:divergenceMarty}
 \end{gather}
 Increasing $r>0$ slightly and taking a subsequence if necessary,
 the limit $b:=\lim_{k\to\infty}z_k$
 exists in $\bD(a,r)$, and then
 by $\diam_{\delta}(V)<\infty$, we also have $\lim_{k\to\infty}w_k=b$.

 Suppose contrary that $\mathcal{F}$ is locally equicontinuous on $D$.
 Then by the Arzel\`a--Ascoli theorem, taking a subsequence if necessary,
 there exists a locally uniform limit $f:=\lim_{k\to\infty}f_k$ on $D$.
 Then there exists $\epsilon>0$ so small that for every $k\in\bN$,
 $f_k(\bD(b,\epsilon))\subset U_{f(b)}$.
 Hence, for every $k\in\bN$ so large that $z_k,w_k\in\bD(b,\epsilon)$, we have
\begin{gather}
 \delta(f_k(z_k),f_k(w_k))\le d_{U_{f(b)}}(f_k(z_k),f_k(w_k))
 \le d_{\bD(b,\epsilon)}(z_k,w_k)\label{eq:decrease}
\end{gather} 
 by the property \eqref{item:neighborhood} of $U_x$ for $x=f(b)$
 in Theorem \ref{th:Yamanoi} and the non-increasing property
 under holomorphisms of the Kobayashi (pseudo)metrics
 (cf.\ Fact \ref{th:decreasing}).
 By \eqref{eq:decrease} and \eqref{eq:comparable}, we have
 \begin{multline*}
  \lim_{k\to\infty}\frac{\delta(f_k(z_k),f_k(w_k))}{d_{\bD(a,r)}(z_k,w_k)}
  \le\lim_{k\to\infty}\frac{d_{\bD(b,\epsilon)}(z_k,w_k)}{d_{\bD(a,r)}(z_k,w_k)}\\
  \le\lim_{k\to\infty}\left(\frac{r}{|z_k-w_k|}\cdot\frac{\epsilon\cdot|z_k-w_k|}{\epsilon^2-(\epsilon/2)^2}\right)
  =\frac{r}{3\epsilon/4}<\infty,
 \end{multline*}
 which contradicts \eqref{eq:divergenceMarty}.
 Now the proof of the ``only if'' part 
 is complete.
\end{proof}

\begin{theorem}[a Zalcman-type theorem]\label{th:ZB}
 Let $D$ be a domain in $\bC$, $V$ a compact complex space
 equipped with a metric $\delta$
 satisfying the conditions in Theorem $\ref{th:Yamanoi}$,
 and $\mathcal{F}$ a family in $\mathcal{O}(D,V)$. Then, $\mathcal{F}$
 is {\rm not} normal at a point $a\in D$, that is,
 not normal on any open neighborhood of $a$ in $D$, if and only if
 there are sequences $(f_k)$, $(z_k)$, and $(\rho_k)$ in $\mathcal{F}$,
 $D$, and $(0,\infty)$, respectively, and
 a non-constant $g\in\mathcal{O}(\bC,V)$ 
 such that 
 $\lim_{k\to\infty}z_k=a$, that $\lim_{k\to\infty}\rho_k=0$, and that
 $\lim_{k\to\infty}f_k(z_k+\rho_k v)= g(v)$
 locally uniformly on $\bC$. 
\end{theorem}

We will divide the proof of Theorem \ref{th:ZB} into two parts:
a deduction of Theorem \ref{th:ZB} from the {\itshape original} Zalcman's lemma
(Theorem \ref{th:Zalcman} below),
and a proof of this original one. 
The former part, say, a {\itshape double rescaling argument},
also motivates a part of the proof of Theorem \ref{th:rescaling}.

The following is due to Zalcman 
in the case where $V=\bP^1$ equipped with the Fubini-Study metric.

\begin{theorem}\label{th:Zalcman}
 Let $V$ be a compact complex space
 equipped with a metric $\delta$
 satisfying the conditions in Theorem $\ref{th:Yamanoi}$,
 and $\mathcal{F}$ a family in $\mathcal{O}(\bD,V)$. Then,
 $\mathcal{F}$ is not normal on $\bD$ if and only if
 there are sequences $(f_k)$, $(z_k)$, and $(\rho_k)$ in $\mathcal{F}$,
 $\bD$, and $(0,\infty)$, respectively, and
 a non-constant $g\in\mathcal{O}(\bC,V)$ 
 such that $\lim_{k\to\infty}z_k$ exists in $\bD$,
 that $\lim_{k\to\infty}\rho_k=0$, and that
 $\lim_{k\to\infty}f_k(z_k+\rho_k v)=g(v)$ locally uniformly on $\bC$.
\end{theorem}

\begin{proof}[Proof of Theorem $\ref{th:ZB}$ using Theorem $\ref{th:Zalcman}$]
 The ``if'' part of Theorem $\ref{th:ZB}$
 is straightforward from that of the original Zalcman's lemma
 (Theorem \ref{th:Zalcman}).

 Suppose that $\mathcal{F}$ 
 is not normal at a point $a\in D$. Then by \eqref{eq:comparable}
 and the Arzel\`a--Ascoli theorem,
 $\limsup_{r\to 0}\sup_{f\in\mathcal{F}}L_{f,\bD(a,r)}=\infty$,
 that is, there are sequences $(r_k)$ in $(0,\infty)$
 and $(f_k)$ in $\mathcal{F}$ such that $\lim_{k\to\infty}r_k=0$ and that
 \begin{gather}
 \lim_{k\to\infty}L_{f_k,\bD(a,r_k)}=\infty.\label{eq:divergence}
 \end{gather}

 Fix $\epsilon>0$ small enough. Then
 for every $k\in\bN$, a holomorphic mapping
 $g_k\in\mathcal{O}(\bD(1+\epsilon),V)$ is defined by
 \begin{gather*}
 g_k(w):=f_k(a+r_k w),
 \end{gather*}
 and then $L_{g_k,\bD}=L_{f_k,\bD(a,r_k)}$ by \eqref{eq:invariance}. 
 Hence by \eqref{eq:divergence} and a Marty-type theorem (Theorem \ref{th:Marty}),
 the family $\{g_k:k\in\bN\}$ is not normal on $\bD(1+\epsilon/2)$,
 so that by Theorem \ref{th:Zalcman}, there are sequences $(w_j),(\eta_j)$, and
 $(k_j)$ in $\bC,(0,\infty)$, and $\bN$, respectively, and a
 non-constant $g\in\mathcal{O}(\bC,V)$ such that $\lim_{j\to\infty}w_j$
 exists in $\bD(1+\epsilon/2)$, that $\lim_{j\to\infty}\eta_j=0$, that
 $\lim_{j\to\infty}k_j=\infty$, and that
 \begin{gather*}
 g(v)=\lim_{j\to\infty}g_{k_j}(w_j+\eta_jv)
 \end{gather*} 
 locally uniformly on $\bC$. For every $j\in\bN$, we have
 $g_{k_j}(w_j+\eta_jv)=f_{k_j}((a+r_{k_j}w_j)+(r_{k_j}\eta_j)v)$
 on $\bD(\eta_j^{-1}(\epsilon/2))$,
 and set $z_j:=a+r_{k_j}w_j$ and $\rho_j:=r_{k_j}\eta_j$.

 Now the proof of the ``only if'' part of Theorem $\ref{th:ZB}$
 is complete
 since $\lim_{j\to\infty}z_j=a$ and
 $\lim_{j\to\infty}\rho_j=0$.
\end{proof}

\begin{proof}[Proof of Theorem $\ref{th:Zalcman}$] 
 If $\mathcal{F}$ is not normal on $\bD$, then by \eqref{eq:comparable}
 and the Arzel\`a--Ascoli theorem,
 we have $\sup_{f\in\mathcal{F}}L_{f,\bD(r)}=\infty$
 for some $r\in(0,1)$.
 Increasing $r\in(0,1)$ slightly if necessary,
 we can choose sequences $(z_k'),(w_k')$ in $\bD(r)$
 and a sequence $(f_k)$ in $\mathcal{F}$ such that both $\{z_k':k\in\bN\}$
 and $\{w_k':k\in\bN\}$ are relatively compact in $\bD(r)$,
 that $z_k'\neq w_k'$ for every $k\in\bN$, 
 and that
 $\lim_{k\to\infty}\delta(f_k(z_k'),f_k(w_k'))/d_{\bD(r)}(z_k',w_k')=\infty$.
 Then, by \eqref{eq:comparable}, we also have
 $\lim_{k\to\infty}\delta(f_k(z_k'),f_k(w_k'))/|z_k'-w_k'|=\infty$.

 For every $k\in\bN$, set
 \begin{gather*}
  M_k:=\sup_{z,w\in\bD(r):z\neq w}\left(\frac{r^2-|z|^2}{r^2}\cdot\frac{\delta(f_k(z),f_k(w))}{|z-w|}\right).
 \end{gather*}
 Then $\lim_{k\to\infty}M_k=\infty$.
 For every $k\in\bN$,
 choose distinct $z_k,w_k\in\bD(r)$ satisfying that
 \begin{gather}
   \frac{r^2-|z_k|^2}{r^2}\cdot
\frac{\delta(f_k(z_k),f_k(w_k))}{|z_k-w_k|}\ge\frac{M_k}{2},\label{eq:large}
 \end{gather}
 and set
 \begin{gather*}
  \rho_k:=\frac{|z_k-w_k|}{\delta(f_k(z_k),f_k(w_k))}\left(\le\frac{2}{M_k}\right)
  \quad\text{and}\quad R_k:=\frac{r-|z_k|}{\rho_k}\left(\ge\frac{M_k}{2}\cdot\frac{r}{2}\right).
 \end{gather*}
 Then $\rho_k\to 0$ and $R_k\to\infty$ as $k\to\infty$. Taking a subsequence
 if necessary, we can also assume that $\lim_{k\to\infty}z_k$ exists in
 $\overline{\bD(r)}$.

 For every $k\in\bN$ and every $v\in\bD(R_k)$, $z_k+\rho_k v\in\bD(r)$. 
 Hence for each $k\in\bN$, a holomorphic mapping $g_k\in\mathcal{O}(\bD(R_k),V)$
 is defined by
 \begin{gather*}
  g_k(v):=f_k(z_k+\rho_k v)
 \end{gather*}
 and satisfies that
 for every distinct $x,y\in\bD(R_k)$,
 \begin{multline*}
 \frac{\delta(g_k(x),g_k(y))}{|x-y|}
 =\frac{\delta(f_k(z_k+\rho_k x),f_k(z_k+\rho_k y))}{|(z_k+\rho_k x)-(z_k+\rho_k y)|}\cdot\rho_k\\
  \le \frac{r^2M_k}{r^2-|z_k+\rho_k x|^2}\cdot
  \frac{2}{M_k}\frac{r^2-|z_k|^2}{r^2}
  =2\frac{r^2-|z_k|^2}{r^2-|z_k+\rho_k x|^2}\\
  \le 2\frac{(r+|z_k|)(r-|z_k|)}{(r+|z_k+\rho_k x|)(r-|z_k|-\rho_k|x|)}
  \le 2\cdot\frac{2r}{r}\cdot \frac{1}{1-|x|/R_k}
  = \frac{4}{1-|x|/R_k}.
 \end{multline*}
 Hence $\{g_k:k\in\bN\}$ is locally equicontinuous on $\bC$.

 By the Arzel\`a--Ascoli theorem,
 taking a subsequence if necessary, there exists the locally uniform limit
 $g:=\lim_{k\to\infty}g_k$ on $\bC$, which is in $\mathcal{O}(\bC,V)$.
 It remains to show that $g$ is non-constant: For every $k\in\bN$, set
\begin{gather*}
  v_k:=\frac{z_k-w_k}{\rho_k}\in\bC\setminus\{0\}.
\end{gather*} 
 By the definition of $\rho_k$, we have not only
 $\sup_{k\in\bN}|v_k|\le\diam_{\delta}V<\infty$
 but also, for every $k\in\bN$,
 \begin{gather}
  \frac{\delta(g_k(0),g_k(v_k))}{|v_k|}
=\frac{\delta(f_k(z_k),f_k(w_k))}{|z_k-w_k|}\cdot\rho_k=1.\label{eq:closing}
 \end{gather}
 Taking a subsequence if necessary, we can assume that the limit
 $v^*:=\lim_{k\to\infty}v_k$ exists in $\bC$.

 Suppose contrary that $g$ is constant on $\bC$. Then
 by \eqref{eq:closing}, we must have $v^*=0$.
 Fix $R>1$. For every $k\in\bN$ large enough, we have $v_k\in\bD(R)$ and
 $g_k(\bD(R))\subset U_{g(0)}$, so that
 \begin{gather}
  \delta(g_k(0),g_k(v_k))
  \le d_{U_{g(0)}}(g_k(0),g_k(v_k))\le d_{\bD(R)}(0,v_k)\label{eq:shrinking}
 \end{gather}
 by the property \eqref{item:neighborhood} of $U_x$ for $x=g(0)$
 in Theorem \ref{th:Yamanoi} and the non-increasing property
 under holomorphisms
 of the Kobayashi (pseudo)metrics (cf.\ Fact \ref{th:decreasing}).
 By \eqref{eq:shrinking} and \eqref{eq:comparable}, we have
 \begin{multline*}
  \limsup_{k\to\infty}\frac{\delta(g_k(0),g_k(v_k))}{|v_k|}
  \le\limsup_{k\to\infty}\frac{d_{\bD(R)}(0,v_k)}{|v_k|}\\
  \le\limsup_{k\to\infty}\frac{R}{R^2-|v_k|^2}
  =\frac{1}{R}<1,
 \end{multline*}
 which contradicts \eqref{eq:closing}. Hence 
 $g$ is non-constant on $\bC$, and
 the proof of the ``only if'' part is complete.

 Suppose now that 
 there are sequences $(f_k)$, $(z_k)$, and $(\rho_k)$ in $\mathcal{F}$,
 $\bD$, and $(0,\infty)$, respectively, and
 a non-constant $g\in\mathcal{O}(\bC,V)$ 
 such that the limit $a:=\lim_{k\to\infty}z_k$ exists in $\bD$,
 that $\lim_{k\to\infty}\rho_k=0$, and that
 $\lim_{k\to\infty}f_k(z_k+\rho_k v)=g(v)$ locally uniformly on $\bC$.
 Fix $r>0$ so small that $\bD(a,r)\Subset\bD$.

 Suppose contrary that $\mathcal{F}$ is normal on $\bD$. Then by
 a Marty-type theorem (Theorem \ref{th:Marty}), we have
 $\sup_{f\in\mathcal{F}}L_{f,\bD(a,r)}<\infty$.
 In particular, by \eqref{eq:comparable},
 we have
\begin{gather*}
 C:=\sup_{f\in\mathcal{F}}\left(\sup_{z,w\in\bD(a,r/2),z\neq w}\frac{\delta(f(z),f(w))}{|z-w|}\right)<\infty.
\end{gather*}
 For every distinct $x,y\in\bC$,
 since $\lim_{k\to\infty}(z_k+\rho_k x)=\lim_{k\to\infty}(z_k+\rho_k y)=a$,
 we have
 \begin{gather*}
  \frac{\delta(g(x),g(y))}{|x-y|}=\lim_{k\to\infty}\frac{\delta(f_k(z_k+\rho_k x),f_k(z_k+\rho_k y))}{|(z_k+\rho_k x)-(z_k+\rho_k y)|}\cdot\rho_k\le C\cdot\lim_{k\to\infty}\rho_k=0,
 \end{gather*}
 so that $g(x)=g(y)$. This contradicts that $g$ is non-constant on $\bC$,
 and the proof of the ``if'' part is also complete.
\end{proof}

\begin{remark}
 For a much shorter proof of (the ``only if'' part of) Theorem \ref{th:ZB} in the case where $V$ is
 an Hermitian compact complex manifold,
 see Berteloot \cite[\S 2.1]{Berteloot06}.
\end{remark}

\begin{remark}\label{th:Brodyproof}
 Brody's theorem (cf.\ Fact \ref{th:little})
 in the setting that $U$ is not necessarily an Hermitian compact complex manifold
 follows from the original Zalcman's lemma (Theorem \ref{th:Zalcman})
 by showing that {\itshape if a compact complex space $V$ is
 not Kobayashi hyperbolic,
 then the family $\mathcal{O}(\bD,V)$ is not normal on $\bD$}:
 Fix $p,q\in V$. We can fix $r\in(0,1)$ so close to $1$ that
 $d_V(p,q)=\inf\sum_{j=1}^md_{\bD}(a_j,b_j)$, where the infimum is taken over
 all such chains $((f_j)_{j=1}^m,((a_j,b_j))_{j=1}^m)$
 of holomorphic disks from $p$ to $q$ that satisfy
 $a_j\equiv 0$ and $b_j\in\bD(r)$ for every $j\in\{1,2,\ldots,m\}$.
 Let us equip $V$ with a metric $\delta$
 satisfying the conditions in Theorem $\ref{th:Yamanoi}$.

 Suppose now that $\mathcal{O}(\bD,V)$ is normal on $\bD$.
 Then by a Marty-type theorem (Theorem \ref{th:Marty}), we have
 $C(r):=\sup_{f\in\mathcal{O}(\bD,V)}L_{f,\bD(r)}<\infty$, and then 
 using also the triangle inequality for $\delta$, we have
 $d_V(p,q)\ge\delta(p,q)/C(r)$. This implies that $d_V(p,q)>0$ if $p\neq q$,
 i.e., $V$ is Kobayashi hyperbolic.
\end{remark}

\begin{remark}
For another proof of 
Brody's theorem (cf.\ Fact \ref{th:little})
similar to the original Brody's argument
and  
in the setting that $U$ is not necessarily an Hermitian compact complex manifold,
see Yamanoi \cite[\S2.3]{Yamanoisurvey}.
\end{remark}

\section{Proof of Theorem \ref{th:HeinonenRossi}}
\label{sec:LV}

Let $V$ be a complex space equipped with a metric $\delta$
inducing the topology of $V$,
and $f:\bD\setminus\{0\}\to V$ 
a holomorphic mapping having an isolated essential
singularity at the origin. 
Suppose that
$\bigcap_{r>0}\overline{f(\bD(r)\setminus\{0\})}\neq\emptyset$.
Then we can fix a sequence $(z_n)$ in $\bD\setminus\{0\}$ tending to $0$ as
$n\to\infty$ such that
the limit $a:=\lim_{n\to\infty}f(z_n)$ exists in $V$.
Fix an open neighborhood $W$ of $a$ in $V$ equivalent
to an analytic subset in an open subset $\Omega$ in $\bC^d$ for some $d\in\bN$,
and fix a subdomain $W'\Subset W$ containing $a$. 

If $\liminf_{n\to\infty}\diam_{\delta}f(\partial\bD(|z_n|))>0$, then
we are done.
So, suppose that $\liminf_{n\to\infty}\diam_{\delta}f(\partial\bD(|z_n|))=0$.
Taking a subsequence if necessary, we can even assume that
\begin{gather}
 \lim_{n\to\infty}\diam_{\delta}f(\partial\bD(|z_n|))=0.\label{eq:shrink}
\end{gather}
Then for every $n\in\bN$ large enough,
$f(\partial\bD(|z_n|))\subset W'$.
For every $n\in\bN$ large enough, 
since the origin is an isolated essential singularity
of $f$, by Riemann's extension theorem, the following maximum
\begin{gather*}
 r_n':=\max\{r\in(0,|z_n|):f(\partial\bD(r))\not\subset W'\}>0
\end{gather*}
exists, and then $f(\overline{\bD(|z_n|)}\setminus\bD(r_n'))\subset\overline{W'}$
 by the continuity of $f$.
 Fix a sequence $(z_n')$ in $\bD\setminus\{0\}$ tending to $0$
 as $n\to\infty$ such that 
 for every $n\in\bN$ large enough, 
 $z_n'\in\partial\bD(r_n')$ and $f(z_n')\in\overline{W'}\setminus W'$.
 By the compactness of $\overline{W'}\setminus W'$, taking a subsequence
 if necessary,
 the limit $b:=\lim_{n\to\infty}f(z_n')$
 exists in $\overline{W'}\setminus W'$.
 It remains to show
 that $\liminf_{n\to\infty}\diam_{\delta}f(\partial\bD(|z_n'|))>0$.

 Suppose contrary that $\liminf_{n\to\infty}\diam_{\delta}f(\partial\bD(|z_n'|))=0$.
 Taking a subsequence if necessary, we can even assume that
\begin{gather}
 \lim_{n\to\infty}\diam_{\delta}f(\partial\bD(|z_n'|))=0.\tag{\ref{eq:shrink}'}\label{eq:shrinkanother}
\end{gather}
 Since $a$ and $b$ are distinct points in $W$, which we identify
 with an analytic subset in an open subset $\Omega$ in $\bC^d$,
 there exists
 an affine coordinate system $w=(w_1,\ldots,w_d)$ on $\Omega$ such that
 $w(a)=0$ and that $w_1(b)\neq 0$. Set
 $w\circ f=(f_1,\ldots,f_d):f^{-1}(W)\to w(W)$.
 Then, for every $n\in\bN$ large enough, under the assumptions \eqref{eq:shrink}
 and \eqref{eq:shrinkanother},
 we have both
 \begin{gather}
 f_1(\partial\bD(|z_n|))\subset\bD(|w_1(b)|/3)\quad\text{and}\quad
 f_1(\partial\bD(|z_n'|))\subset\bD(w_1(b),|w_1(b)|/3).\label{eq:outside}
 \end{gather}
 Fix such $n\in\bN$ as satisfies \eqref{eq:outside}.
 Let $\ell$ be a line segment in the ring domain
 $\bD(|z_n|)\setminus\overline{\bD(|z_n'|)}(\Subset f^{-1}(W))$ 
 having one end point in $\partial\bD(|z_n|)$ and 
 the other in $\partial\bD(|z_n'|)$.
 Then the path $f_1(\ell)$ in $w_1(W)$ joins the closed curves
 $f_1(\partial\bD(|z_n|))$ and 
 $f_1(\partial\bD(|z_n'|))$, so by \eqref{eq:outside},
 we may fix $y_0\in\ell$ such that
 \begin{gather}
 f_1(y_0)\not\in \overline{\bD(|w_1(b)|/3)}\cup\overline{\bD(w_1(b),|w_1(b)|/3)}.\label{eq:inside}
 \end{gather}
 Since $f_1$ is a holomorphic function on $f^{-1}(W)$ and takes the value $f_1(y_0)$
 at least at $y_0\in\bD(|z_n|)\setminus\overline{\bD(|z_n'|)}$, 
 by the residue theorem,
\begin{multline}
 1\le\frac{1}{2i\pi}\int_{\partial(\bD(|z_n|)\setminus\overline{\bD(|z_n'|)})}\frac{f_1'(z)\rd z}{f_1(z)-f(y_0)}\\
=\frac{1}{2i\pi}\int_{(f_1)_*(\partial\bD(|z_n|))}\frac{\rd w_1}{w_1-f(y_0)}
  -\frac{1}{2i\pi}\int_{(f_1)_*(\partial\bD(|z_n'|))}\frac{\rd w_1}{w_1-f(y_0)},\label{eq:argument}
\end{multline}
where the boundary
$\partial(\bD(|z_n|)\setminus\overline{\bD(|z_n'|)})$ is canonically oriented.
On the other hand, 
by \eqref{eq:outside} and \eqref{eq:inside}, the argument principle
yields
 \begin{gather*}
\frac{1}{2i\pi}\int_{(f_1)_*(\partial\bD(|z_n|))}\frac{\rd w_1}{w_1-f(y_0)}
=\frac{1}{2i\pi}\int_{(f_1)_*(\partial\bD(|z_n'|))}\frac{\rd w_1}{w_1-f(y_0)}=0,
 \end{gather*} 
which contradicts \eqref{eq:argument}. 

Hence 
$\liminf_{n\to\infty}\diam_{\delta}f(\partial\bD(|z_n'|))>0$,
and the proof is complete. \qed

\begin{remark}
 The final residue theoretic argument applied to $f_1$
 can be replaced by a more topological
 argument (for $f_1$) as in \cite[Proof of Lemma 3.1]{OPrescaling}. In 
 \cite[Lemma 3.1]{OPrescaling}, the target Riemannian $n$-manifold $M$ of a
 quasiregular mapping $f:\bB^n\setminus\{0\}\to M$ was assumed to be compact,
 but this assumption can be relaxed as $\bigcap_{r>0}\overline{f(\bB^n(r)\setminus\{0\})}\neq\emptyset$ as in Theorem \ref{th:HeinonenRossi}. Moreover, in 
 \cite[Lemma 3.1]{OPrescaling}, we only claimed that
 $\limsup_{r\to 0}\diam(f(\partial\bB^n(r)))>0$, but this assertion can be
 so strengthen that there exists a sequence $(x_j)$ in $\bB^n\setminus\{0\}$ 
 tending to $0$ as $j\to\infty$ such that $\lim_{j\to\infty}f(x_j)$ exists in $M$
 and that $\liminf_{j\to\infty}\diam(f(\partial\bB^n(|x_j|)))>0$, 
 as in Theorem \ref{th:HeinonenRossi}.
\end{remark} 

\section{Proof of Theorem \ref{th:PicardKobayashi}}
\label{sec:Picard}

Recall that
$d_V$ denotes the Kobayashi pseudometric on a complex space $V$ and
that, for every metric $\delta$ on $V$, we set
$\diam_{\delta}(S)=\sup\{\delta(a,a'):a,a'\in S\}$
for a non-empty subset $S$ in $V$. 

Let $Y$ be a complex subspace in a complex space $Z$, fix a metric 
$\delta$ on $Z$ inducing the topology of $Z$,
and let $f:\bD\setminus\{0\}\to Y$ be a holomorphic mapping. 

Suppose that $\bigcap_{r>0}\overline{f(\bD(r)\setminus\{0\})}\neq\emptyset$
(as a subset in $Z$)
and that $f$ does not extend to a holomorphic mapping from $\bD$ to $Z$.
We claim that $Y$ is not a hyperbolically imbedded complex subspace in $Z$.

Under the above assumption, by Theorem \ref{th:HeinonenRossi}, there exists
a sequence $(z_n)$ in $\bD\setminus\{0\}$ tending to
$0$ as $n\to\infty$ such that the limit
\begin{gather*}
 a:=\lim_{n\to\infty}f(z_n)
\end{gather*}
exists in $\overline{Y}$ and that 
$r_0:=\liminf_{n\to 0}\diam_{\delta}(f(\partial\bD(|z_n|)))>0$.
Fix a relatively compact open neighborhood $W$ of $a$ in
$\{p\in Z:\delta(p,a)<r_0\}$. Then,
taking a subsequence of $(z_n)$ if necessary, 
there is a sequence $(w_n)$ in $\bD\setminus\{0\}$ such that
for every $n\in\bN$,
\begin{gather*}
 w_n\in\partial\bD(|z_n|)
\end{gather*}
and $f(w_n)\in\partial W$. By the compactness of $\partial W$,
taking a subsequence of $(z_n)$ if necessary, the limit 
\begin{gather*}
 b:=\lim_{n\to\infty}f(w_n)
\end{gather*}
also exists in $\overline{Y}\setminus\{a\}$.

Now, for every open neighborhoods
$U_a,U_b$ of $a,b$ in $Z$, respectively, we have  
\begin{multline*}
 d_Y(Y\cap U_a,Y\cap U_b)\le\limsup_{n\to\infty}d_Y(f(z_n),f(w_n))\\
\le\limsup_{n\to\infty}d_{\bD\setminus\{0\}}(z_n,w_n)
\le\limsup_{n\to\infty}\diam_{d_{\bD\setminus\{0\}}}(\partial\bD(|z_n|))=0,
\end{multline*}
where the second inequality is by the non-increasing property of
Kobayashi pseudometrics under holomorphisms (see Fact \ref{th:decreasing}) 
and the final equality is by a direct computation
(cf.\ Example \ref{th:puncture}).

Hence $Y$ is not a hyperbolically imbedded complex subspace in $Z$,
and the proof is complete. \qed

\section{Proof of Theorem \ref{th:rescaling}}
\label{sec:rescaling}
Let $V$ be a compact complex space and
$f:\bD\setminus\{0\}\to V$ be a holomorphic mapping 
having an isolated essential singularity at the origin, and
fix a metric $\delta$ on $V$ satisfying the conditions
in Theorem \ref{th:Yamanoi}.

Let us first show the ``only if'' part of Theorem \ref{th:rescaling}.
We study the cases that {\bfseries (i)}
$\limsup_{z\to 0}L_{f,\bD(z,|z|/2)}=\infty$ and
that {\bfseries (ii)} $\limsup_{z\to 0}L_{f,\bD(z,|z|/2)}<\infty$, separately.

Suppose first that $\limsup_{z\to 0}L_{f,\bD(z,|z|/2)}=\infty$.
In this case, the following ``double rescaling'' argument is similar to
that in the proof of Theorem \ref{th:ZB} using Theorem \ref{th:Zalcman}
in Section \ref{sec:background}.
Choose a sequence $(y_k)$ in $\bD\setminus\{0\}$ such that
$\lim_{k\to\infty}y_k=0$ and that
\begin{gather}
 \lim_{k\to\infty}L_{f,\bD(y_k,|y_k|/2)}=\infty.\label{divergenceagain}
\end{gather}
Fix $\epsilon>0$ small enough. Then for every $k\in\bN$,
a holomorphic mapping $g_k:\bD(1+\epsilon)\to V$ is defined by
\begin{gather*}
 g_k(w):= f\left(y_k+\frac{|y_k|}{2}w\right),
\end{gather*}
so that $L_{g_k,\bD}=L_{f,\bD(y_k,|y_k|/2)}$ by \eqref{eq:invariance}.
Hence by \eqref{divergenceagain} and a Marty-type theorem (Theorem \ref{th:Marty}),
the family $\{g_k:k\in\bN\}$ is not normal on $\bD(1+\epsilon/2)$, so that 
by Theorem \ref{th:Zalcman}, there are
sequences $(w_j)$, $(\eta_j)$, and $(k_j)$ in 
$\bC$, $(0,\infty)$, and $\bN$, respectively, 
and a non-constant $g\in\mathcal{O}(\bC,V)$ 
such that $\lim_{j\to\infty}w_j$ exists in $\bD(1+\epsilon/2)$,
that $\lim_{j\to\infty}\eta_j=0$, that $\lim_{j\to\infty}k_j=\infty$, and that
\begin{gather*}
 g(v)=\lim_{j\to\infty}g_{k_j}(w_j + \eta_j v)
\end{gather*}
locally uniformly on $\bC$. For every $j\in\bN$, we have
$g_{k_j}(w_j+\eta_j v)=f((y_{k_j}+(|y_{k_j}|/2)w_j)+((|y_{k_j}|/2)\eta_j)v)$
on $\bD(\eta_j^{-1}(\epsilon/2))$, and set $z_j:=y_{k_j}+(|y_{k_j}|/2)w_j$ and
$\rho_j:=(|y_{k_j}|/2)\eta_j$. We are done
in this case (i)
since $\lim_{j\to\infty}z_j=0$ and $\lim_{j\to\infty}\rho_j=0$.

Suppose next that $\limsup_{z\to 0}L_{f,\bD(z,|z|/2)}<\infty$. 
By the compactness of $V$, 
we have $\bigcap_{r>0}\overline{f(\bD(r)\setminus\{0\})}\neq\emptyset$,
so that by Theorem \ref{th:HeinonenRossi},
there is a sequence $(w_k)$ in $\bD\setminus\{0\}$ 
tending to $0$ as $k\to\infty$ such
that $a:=\lim_{k\to\infty}f(w_k)$ exists in $V$
and that
$\liminf_{k\to\infty}\diam_{\delta}f(\partial\bD(|w_k|))>0$. 

For every $k\in\bN$, a holomorphic mapping
$g_k:\bD(|w_k|^{-1})\setminus\{0\}\to V$ is defined by 
\begin{gather*}
 g_k(v):= f(|w_k|v).
\end{gather*}
Then for every $v\in\bC\setminus\{0\}$, 
\begin{gather*}
\limsup_{k\to\infty}L_{g_k,\bD(v,|v|/2)}
=\limsup_{k\to\infty}L_{f,\bD(|w_k|v,|w_k|v/2)}
\le\limsup_{z\to 0}L_{f,\bD(z,|z|/2)}<\infty,
\end{gather*}
where the equality is by \eqref{eq:invariance} and the next
inequality is by $\lim_{k\to\infty}|w_k|=0$.
Hence by \eqref{eq:comparable}
 and the Arzel\`a--Ascoli theorem,
the family $\{g_k:k\in\bN\}$ is normal
on $\bC\setminus\{0\}$ so,
taking a subsequence if necessary, the locally uniform limit 
$g:=\lim_{k\to\infty}g_k$ exists on $\bC\setminus\{0\}$,
which is in $\mathcal{O}(\bC\setminus\{0\},V)$. It remains to show that
$g$ is non-constant on $\bC\setminus\{0\}$. If $g$ is constant, then
since for every $k\in\bN$, 
\begin{gather*}
 g_k(\partial\bD)=f(\partial\bD(|w_k|))\ni f(w_k), 
\end{gather*}
we have $g\equiv a=\lim_{k\to\infty}f(w_k)$ on $\bC\setminus\{0\}$,
and in turn 
\begin{gather*}
0=\diam_{\delta}(\{a\})=\limsup_{k\to\infty}\diam_{\delta}
(g_k(\partial\bD))=\limsup_{k\to\infty}\diam_{\delta}f(\partial\bD(|w_k|)),
\end{gather*}
which contradicts that $\liminf_{k\to\infty}\diam_{\delta}f(\partial\bD(|w_k|))>0$.
Hence $g$ is non-constant on $\bC\setminus\{0\}$,
and we are also done in this case (ii)
by setting $z_k\equiv 0$ and $\rho_k:=|w_k|$ for each $k\in\bN$. 
Now the proof of the ``only if'' part of Theorem \ref{th:rescaling} is complete.

Suppose now that there are sequences $(z_k)$ and $(\rho_k)$ 
in $\bC$ and $(0,\infty)$, respectively, 
and a non-constant holomorphic mapping $g:X\to V$, 
where $X$ is either $\bC$ or $\bC\setminus \{0\}$, 
such that $\lim_{k\to\infty}z_k=0$ in $\bC$, 
that $\lim_{k\to\infty}\rho_k=0$ in $\bR$, and
that $\lim_{k\to\infty}f(z_k+\rho_k v)= g(v)$
locally uniformly on $X$. If $f$ extends to a holomorphic mapping from $\bD$ to $V$,
then for every $v\in X$, $g(v)=\lim_{k\to\infty}f(z_k+\rho_k v)=f(0)$,
which contradicts that $g$ is non-constant on $X$. Now the proof of
the ``if'' part of Theorem \ref{th:rescaling} is also complete. \qed

\begin{remark}
The proof of Theorem \ref{th:rescaling} is similar to 
\cite[Proof of Theorem 1]{OPrescaling} for quasiregular mappings. 
In the holomorphic curve case, however,
the {\itshape Lipschitz continuity on disks} \eqref{eq:Lipschitz}
of holomorphic curves and the invariance of the Kobayashi pseudometrics under
biholomorphisms between complex spaces make the argument much simpler
than that in the quasiregular case.
\end{remark}

\section{On Remark \ref{th:Julia}}
\label{sec:Julia}

Let $V$ be a compact Hermitian manifold equipped with an Hermitian metric
$\delta_V$. Then $\delta_V$ satisfies the properties in Theorem \ref{th:Yamanoi}
(cf.\ \cite[\S2.3]{Yamanoisurvey}). 
Let $f:\bD\setminus\{0\}\to V$
be a holomorphic curve having an isolated essential singularity at the origin, 
and recall that $f^\#(z):=\lim_{w\to z}\delta_V(f(z),f(w))/|z-w|$ 
on $\bD\setminus\{0\}$. Then $f^*\rd s_V=f^\#|\rd z|$ on $\bD$, where
$\rd s_V$ is the arc-length element of $\delta_V$ on $V$.

Let us see that, 
{\itshape in the proof of the ``only if'' part of Theorem $\ref{th:rescaling}$,
the case $\limsup_{z\to 0}L_{f,\bD(z,|z|/2)}<\infty$
occurs if and only if
$f$ is exceptional in Julia's sense}:
by \eqref{eq:comparable}, for every $z\in\bD(2/3)\setminus\{0\}$, we have
\begin{multline*}
 L_{f,\bD(z,|z|/2)}\ge 
 \lim_{w\to z}\frac{\delta_V(f(z),f(w))}{d_{\bD(z,|z|/2)}(z,w)}\\
 \ge\lim_{w\to z}\frac{f^\#(z)}{(|z|/2)/(|z|/2)^2-|w-z|^2)}=\frac{|z|}{2}\cdot f^\#(z).
\end{multline*}
Hence
$f$ is exceptional in Julia's sense if $\limsup_{z\to 0}L_{f,\bD(z,|z|/2)}<\infty$.

If $\limsup_{z\to 0}L_{f,\bD(z,|z|/2)}=\infty$,
then there are sequences $(z_n),(z_n')$, and $(z_n'')$
in $\bC\setminus\{0\}$
such that $\lim_{n\to\infty}z_n=0$,
that $\lim_{n\to\infty}L_{f,\bD(z_n,|z_n|/2)}=\infty$,
and that for every $n\in\bN$,
the points $z_n',z_n''$ are in $\bD(z_n,|z_n|/2)$ and distinct and satisfies
\begin{gather}
 \frac{1}{2}L_{f,\bD(z_n,|z_n|/2)}
 \le\frac{\delta_V(f(z_n'),f(z_n''))}{d_{\bD(z_n,|z_n|/2)}(z_n',z_n'')}.\label{eq:divergenceFatou}
\end{gather}
For every $n\in\bN$,
there is also $w_n\in\bD(z_n,|z_n|/2)$
on the Euclidean line segment joining $z_n',z_n''$ such that
\begin{gather}
 \delta_V(f(z_n'),f(z_n''))\le f^\#(w_n)|z_n'-z_n''|.\label{eq:average}
\end{gather}
For every $n\in\bN$, we have $|z_n|/2\le|w_n|\le 3|z_n|/2$, and then
by \eqref{eq:divergenceFatou}, \eqref{eq:average}, and \eqref{eq:comparable},
\begin{gather*}
 \frac{1}{2}L_{f,\bD(z_n,|z_n|/2)}
 \le\frac{f^\#(w_n)|z_n'-z_n''|}{d_{\bD(z_n,|z_n|/2)}(z_n',z_n'')}
 \le\frac{f^\#(w_n)}{1/(|z_n|/2)}\le |w_n|f^\#(w_n)
\end{gather*}
so
\begin{gather*}
\limsup_{z\to\infty}|z|f^\#(z)\ge\limsup_{n\to\infty}|w_n|f^\#(w_n)
\ge\frac{1}{2}\limsup_{n\to\infty}L_{f,\bD(z_n,|z_n|/2)}=\infty. 
\end{gather*}
Hence $f$ is not exceptional in Julia's sense. 

\begin{acknowledgement}
 The author thanks the referee for a careful scrutiny, which helped us
 improve this article.
The author also thanks Professor Katsutoshi Yamanoi for discussions on the Kobayashi hyperbolic geometry of complex spaces, and Professor Pekka Pankka for invaluable comments.
\end{acknowledgement}

\def\cprime{$'$}

\end{document}